\documentclass[12pt,leqno]{article}
\usepackage{amsfonts,amsthm,amsmath}
\newtheorem{thm}{Theorem}[section]
\newtheorem{prop}[thm]{Proposition}
\newtheorem{lem}[thm]{Lemma}
\newtheorem{cor}[thm]{Corollary}
\theoremstyle{remark}
\newtheorem{rem}[thm]{Remark}

\newcommand{\ZZ}{\mathbb{Z}}
\newcommand{\RR}{\mathbb{R}}

\DeclareMathOperator{\Aut}{Aut}
\DeclareMathOperator{\wt}{wt}

\begin{document}

\title{Extremal Type~I $\ZZ_k$-codes and 
$k$-frames of odd unimodular lattices}

\author{
Masaaki Harada\thanks{
Research Center for Pure and Applied Mathematics, 
Graduate School of Information Sciences, 
Tohoku University, Sendai 980--8579, Japan.
email: mharada@m.tohoku.ac.jp.
This work 
was partially carried out
at Yamagata University.}
}

\maketitle


\begin{abstract}
For some extremal (optimal) odd unimodular lattice $L$ in
dimensions $12,16,20,28,32,36,40$ and $44$,
we determine all integers $k$ such that $L$ contains a
$k$-frame.
This result yields the existence of 
an extremal Type~I $\ZZ_{k}$-code of lengths 
$12,16,20,32,36,40$ and $44$, and a near-extremal Type~I $\ZZ_k$-code of
length $28$ for positive integers $k$ with only a few exceptions.
\end{abstract}

\section{Introduction}\label{Sec:1}

Self-dual codes and unimodular lattices are studied  from several
viewpoints (see~\cite{SPLAG} for an extensive bibliography). Many
relationships between self-dual codes and unimodular lattices are
known and there are similar situations between two subjects.
As a typical example,
it is known that a unimodular lattice $L$ contains a $k$-frame 
if and only if there is a self-dual $\ZZ_{k}$-code $C$ such that
$L$ is isomorphic to the lattice obtained from $C$ by
Construction A,
where $\ZZ_{k}$ is the ring of integers modulo $k$ with $k \ge 2$.

Type~II $\ZZ_{2k}$-codes were defined in~\cite{BDHO} 
as a class of self-dual codes,
which are related to even unimodular lattices.
If $C$ is a Type~II $\ZZ_{2k}$-code of length $n \le 136$, 
then we have the bound on the minimum Euclidean weight
$d_E(C)$ of $C$ as follows:
$d_E(C) \le 4k \left\lfloor \frac{n}{24} \right\rfloor +4k$
for every positive integer $k$ (see~\cite{HM12}).
A Type~II $\ZZ_{2k}$-code meeting the bound 
with equality is called {extremal} $(n \le 136)$.
It was shown in~\cite{Chapman, GH01} that
the Leech lattice, which is one of the most remarkable lattices,
contains a $2k$-frame for every integer $k \ge 2$.
This result yields the existence of an extremal Type~II $\ZZ_{2k}$-code of
length $24$  for every positive integer $k$.
Recently, the existence of
an extremal Type~II $\ZZ_{2k}$-code of lengths $n=32,40,48,56,64$ 
was established in~\cite{HM12}
for every positive integer $k$.
This was done by finding a $2k$-frame in some
extremal even unimodular lattices in these dimensions $n$.

Recently, it was shown in~\cite{Miezaki} that the odd Leech 
lattice contains a $k$-frame for every integer $k$ with 
$k \ge 3$.
This motivates our investigation of the existence of
a $k$-frame in extremal odd unimodular lattices.
In this paper, 
for some extremal (optimal) odd unimodular lattices $L$ in
dimensions $12,16,20,28,32,36,40$ and $44$,
we determine all integers $k$ such that $L$ contains a
$k$-frame.
This result yields the existence of 
an extremal Type~I $\ZZ_{k}$-code of lengths 
$12,16,20,32,36,40$ and $44$, and a near-extremal Type~I $\ZZ_k$-code of
length $28$ for positive integers $k$ with only a few small exceptions.
This paper is organized as follows. In Section~\ref{sec:Pre}, we
give definitions and some basic properties of self-dual codes
and unimodular lattices used in this paper.
The notion of extremal Type~I $\ZZ_k$-codes of length $n$
is given for $n \le 48$ and $k \ge 2$.
Lemma~\ref{lem:frame} gives a reason why we consider 
unimodular lattices only in  dimension $n$ divisible by $4$.
In Section~\ref{sec:frame}, we provide a method for
constructing $m$-frames in unimodular lattices, which
are constructed from some self-dual $\ZZ_k$-codes by Construction A
(Proposition~\ref{prop:const}).
This method is a slight generalization 
of~\cite[Propositions 3.3 and 3.6]{HM12}.
Using Proposition~\ref{prop:const},
we give $k$-frames in 
the unique extremal odd unimodular lattice
in dimensions $12,16$, 
some extremal (optimal) odd unimodular lattices
in dimensions $20,28, 32,36,40$ and $44$, 
which are listed in Table~\ref{Tab:L}, 
for all integers $k$ satisfying the condition
$(\star)$ in Table~\ref{Tab:L}
(Lemma~\ref{lem:key}).
In Section~\ref{sec:main}, 
several extremal (near-extremal) Type~I $\ZZ_{k}$-codes are 
explicitly constructed
for some positive integers $k$.  
Then we establish the existence of
a $k$-frame in the extremal (optimal) unimodular lattices $L$
in dimensions $12,16,20,28, 32,36$, which are listed in Table~\ref{Tab:L} 
(except only lattices 
$A_3(C_{20,3}(D'_{10}))$ and $A_5(C_{20,5}(D''_{10}))$),
for every integer $k$ with $k \ge \min(L)$, where $\min(L)$
denotes the minimum norm of $L$.
As a consequence, 
by considering the even unimodular neighbors of 
the above extremal odd unimodular lattice in dimension $32$,
it is shown that 
the $32$-dimensional Barnes--Wall lattice $BW_{32}$
contains a $2k$-frame if and only if $k$ is
an integer with $k \ge 2$. 
When $n=40,44$, we show that there is an extremal odd unimodular 
lattice in dimension $n$ containing a $k$-frame
if and only if $k$ is an integer with $k \ge 4$. 
Using the above existence of $k$-frames,
the existence of 
an extremal Type~I $\ZZ_{k}$-code of lengths 
$n=12,16,20,32,36,40, 44$,  
and a near-extremal Type~I $\ZZ_k$-code of
length $n=28$ is established for a positive integer $k$,
where 
$k \ne 1,3$ if $n =32$ and $k \ne 1$ otherwise.
At the end of Section~\ref{sec:main}, we examine the existence of
both $k$-frames in optimal odd unimodular lattices in dimension $48$
and near-extremal Type~I $\ZZ_k$-codes of
length $48$.


All computer calculations in this paper were
done by {\sc Magma}~\cite{Magma}.

%
%

\section{Preliminaries}\label{sec:Pre}

In this section, we give definitions and some basic properties 
of self-dual codes and unimodular lattices 
used in this paper.
The notion of extremal Type~I $\ZZ_k$-codes of length $n$
is given for $n \le 48$ and $k \ge 2$.

\subsection{Self-dual codes}
Let $\ZZ_{k}$ be the ring 
of integers modulo $k$, where $k$ 
is a positive integer. 
In this paper, we always assume that $k\geq 2$ and 
we take the set $\ZZ_{k}$ to be 
$\{0,1,\ldots,k-1\}$.
A $\ZZ_{k}$-code $C$ of length $n$
(or a code $C$ of length $n$ over $\ZZ_{k}$)
is a $\ZZ_{k}$-submodule of $\ZZ_{k}^n$.
A $\ZZ_2$-code and a $\ZZ_3$-code
are called {\em binary} and {\em ternary}, respectively.
The {\em Euclidean weight} of a codeword $x=(x_1,\ldots,x_n)$ of $C$ is
$\sum_{\alpha=1}^{\lfloor k/2 \rfloor}n_\alpha(x) \alpha^2$, 
where $n_{\alpha}(x)$ denotes
the number of components $i$ with $x_i \equiv \pm \alpha \pmod k$ 
$(\alpha=1,2,\ldots,\lfloor k/2 \rfloor)$.
The {\em minimum Euclidean weight} $d_E(C)$ of $C$ is the smallest Euclidean
weight among all nonzero codewords of $C$.

A $\ZZ_{k}$-code $C$ is {\em self-dual} if $C=C^\perp$, where
the dual code $C^\perp$ of $C$ is defined as 
$\{ x \in \ZZ_{k}^n \mid x \cdot y = 0$ for all $y \in C\}$
under the standard inner product $x \cdot y$. 
For only even positive integers $2k$, 
a {\em Type~II} $\ZZ_{2k}$-code was defined in
\cite{BDHO} as a self-dual $\ZZ_{2k}$-code with the property that all
Euclidean weights are congruent to $0$ modulo $4k$.
It is known that a Type~II $\ZZ_{2k}$-code of length $n$ exists
if and only if $n$ is divisible by $8$~\cite{BDHO}.
A self-dual code which is not Type~II is called {\em Type~I}.
Two self-dual $\ZZ_{k}$-codes $C$ and $C'$ are {\em equivalent} 
if there exists a monomial $(\pm 1, 0)$-matrix $P$ with 
$C' = C \cdot P$, where
$C \cdot P = \{ x P\:|\: x \in C\}$.  

\subsection{Unimodular lattices}\label{sec:2U}
A (Euclidean) lattice $L \subset \RR^n$
in dimension $n$
is {\em unimodular} if
$L = L^{*}$, where
the dual lattice $L^{*}$ of $L$ is defined as
$\{ x \in {\RR}^n \mid (x,y) \in \ZZ \text{ for all }
y \in L\}$ under the standard inner product $(x,y)$.
Two lattices $L$ and $L'$ are {\em isomorphic}, denoted $L \cong L'$,
if there exists an orthogonal matrix $A$ with
$L' = L \cdot A$, where $L \cdot A=\{xA \mid x \in L\}$.
The automorphism group $\Aut(L)$ of $L$ is the group of all
orthogonal matrices $A$ with $L = L \cdot A$.
The norm of a vector $x$ is defined as $(x, x)$.
The minimum norm $\min(L)$ of a unimodular
lattice $L$ is the smallest norm among all nonzero vectors of $L$.
The theta series $\theta_{L}(q)$ of $L$ is the formal power
series $\theta_{L}(q) = \sum_{x \in L} q^{(x,x)}$.
The kissing number of $L$ is the second nonzero coefficient of the
theta series.
A unimodular lattice with even norms is said to be {\em even}, 
and that containing a vector of odd norm is said to be {\em odd}.
An even unimodular lattice in dimension $n$
exists if and only
if $n$ is divisible by $8$, while an odd unimodular lattice
exists for every dimension. 

It was shown in~\cite{RS-bound} 
that a unimodular lattice $L$ in dimension $n$
has minimum norm $\min(L) \le 2 \lfloor \frac{n}{24} \rfloor+2$
unless $n=23$ when $\min(L) \le 3$ (see~\cite{Siegel}
for the case that $L$ is even).
A unimodular lattice meeting the bound
with equality is called {\em extremal}.
Any extremal unimodular lattice in dimension $24k$
has to be even~\cite{Gaulter}.
Hence,  an odd unimodular lattice $L$
in dimension $24k$
satisfies $\min(L) \le 2k+1$.
We say that an odd unimodular lattice with the largest minimum norm
among all odd unimodular lattices in that dimension is {\em optimal}.

Let $L$ be a unimodular lattice.
Define $L_0=\{x \in L \mid (x,x) \equiv 0 \pmod 2\}$.
Then $L_0$ is a sublattice of $L$ of index $2$ if $L$ is
odd and $L_0=L$ if $L$ is even.
The {\em shadow} $S$ of $L$ is defined as
$S=L_0^* \setminus L$ if  $L$ is odd 
and as $S=L$ if $L$ is even~\cite{CS-odd}.
Now suppose that $L$ is an odd unimodular lattice.
Then there are cosets $L_1,L_2,L_3$ of $L_0$ such that
$L_0^* = L_0 \cup L_1 \cup L_2 \cup L_3$, where
$L = L_0  \cup L_2$ and $S = L_1 \cup L_3$.
Two lattices are {\em neighbors} if
both lattices contain a sublattice of index $2$
in common.
If the dimension is divisible by $8$,
then $L$ has two even unimodular neighbors
of $L$, namely, $L_0 \cup L_1$ and $L_0 \cup L_3$.


\subsection{Construction A and $k$-frames}

We give a method to construct 
unimodular lattices from self-dual $\ZZ_{k}$-codes, which 
is referred to as {\em Construction A} (see~\cite{{BDHO},{HMV}}). 
If $C$ is a  self-dual $\ZZ_{k}$-code of length $n$, then 
the following lattice 
\[
A_{k}(C)=
\frac{1}{\sqrt{k}}\{(x_1,\ldots,x_n) \in \ZZ^n \mid
(x_1 \bmod k,\ldots,x_n \bmod k)\in C\}
\]
is a unimodular lattice in dimension $n$.
The minimum norm of $A_{k}(C)$ is $\min\{k, d_{E}(C)/k\}$.
Moreover, $C$ is a Type~II $\ZZ_{2k}$-code if and only if
$A_{2k}(C)$ is an even unimodular lattice~\cite{BDHO}.

A set $\{f_1, \ldots, f_{n}\}$ of $n$ vectors $f_1, \ldots, f_{n}$ of 
a unimodular lattice $L$ in dimension $n$ with
$ ( f_i, f_j ) = k \delta_{i,j}$
is called a {\em $k$-frame} of $L$,
where $\delta_{i,j}$ is the Kronecker delta.
It is trivial that if 
a unimodular lattice in dimension $n$
contains a $k$-frame then the number of vectors of norm $k$
is greater than or equal to $2n$.
It is known that a unimodular lattice $L$ contains a $k$-frame 
if and only if there exists a self-dual $\ZZ_{k}$-code $C$ with 
$A_{k}(C) \cong L$ (see~\cite{HMV}).
Therefore, we have the following:

\begin{lem}\label{lem:LtoC}
Suppose that there is a unimodular lattice $L$ in dimension
$n$ containing a $k$-frame.
Then there is a self-dual $\ZZ_k$-code $C$ such that
$d_E(C)$ is greater than or equal to $k \min(L)$.
\end{lem}

The above lemma is useful when establishing the existence of
extremal Type~I $\ZZ_k$-codes in Section~\ref{sec:main}.

By the following lemma, it is sufficient to consider 
the existence of a $p$-frame in a unimodular lattice
for each prime $p$.
The lemma also gives a reason why we consider 
unimodular lattices only in dimension $n$ divisible by $4$.

\begin{lem}[{\cite[Lemma~5.1]{Chapman}}]\label{lem:Chapman}
\label{lem:frame}
Let $n$ be a positive integer divisible by $4$.
If a lattice $L$ in dimension $n$ contains a $k$-frame, then
$L$ contains a $km$-frame for every positive integer $m$.
\end{lem}

\subsection{Upper bounds on the minimum Euclidean weights}

It is known~\cite{MPS,Rains,RS-bound} that
a self-dual $\ZZ_k$-code $C$ of length $n$ satisfies the
following bound:
\begin{equation}\label{eq:kbound}
d_E(C) \le
\begin{cases}
4 \lfloor \frac{n}{24} \rfloor+4 & \text{ if } 
                      k=2, n \not\equiv 22 \pmod{24}, \\
4 \lfloor \frac{n}{24} \rfloor+6 & \text{ if } 
                      k=2, n \equiv 22 \pmod{24}, \\
3 \lfloor \frac{n}{12} \rfloor +3 & \text{ if } k=3, \\
8 \lfloor \frac{n}{24} \rfloor+8 & \text{ if } k=4, 
                                  n \not\equiv 23 \pmod{24}, \\
8 \lfloor \frac{n}{24} \rfloor+12 & \text{ if } k=4, 
                                  n \equiv 23 \pmod{24}.
\end{cases}
\end{equation}
Note that 
a binary self-dual code of length divisible by $24$ meeting the bound
with equality must be Type~II~\cite{Rains}.


Although the following two lemmas are somewhat trivial, 
we give proofs for the sake of completeness.

\begin{lem}\label{lem:bound}
Let $C$ be a self-dual $\ZZ_k$-code of length $n$.
\begin{itemize}
\item[\rm (a)]
If $n \ne 23$ and 
$k \ge 2 \lfloor \frac{n}{24} \rfloor+3$,
then $d_E(C) \le 2k \lfloor \frac{n}{24} \rfloor+2k$.
\item[\rm (b)]
If $n = 23$ and $k \ge 4$, then $d_E(C) \le 3k$.
\end{itemize}
\end{lem}
\begin{proof}
Since both cases are similar,  we only give a proof of (a).
Note that the Euclidean weight of a codeword of $C$ is
divisible by $k$.
Suppose that 
$d_E(C) \ge 2k \lfloor \frac{n}{24} \rfloor+3k$.
Since $\min(A_{k}(C))=\min\{k, d_{E}(C)/k\}$,
$\min(A_{k}(C)) \ge  2 \lfloor \frac{n}{24} \rfloor+3$,
which is a contradiction to the upper bound
on the minimum norms of unimodular lattices.
\end{proof}

\begin{lem}\label{lem:bound2}
If $C$ is a self-dual $\ZZ_k$-code of length $48$,
then $d_E(C) \le 6k$.
\end{lem}
\begin{proof}
By the bound (\ref{eq:kbound}) and Lemma~\ref{lem:bound},
it is sufficient to consider the cases only for $k=5,6$.
Assume that $k=5,6$ and $d_E(C) \ge 7k$.
Since $k < d_E(C)/k$, 
$\min(A_k(C))= k$ and the kissing number of $A_k(C)$ is $96$.
Note that unimodular lattices $L$ with $\min(L)=6$ and $5$ are
extremal even unimodular lattices and optimal odd 
unimodular lattices, respectively.
However, 
the kissing numbers of such lattices
are $52416000$ (see~\cite[Chap.~7, (68)]{SPLAG}) and 
$385024$ or $393216$
\cite{HKMV}, respectively.
This is a contradiction.
\end{proof}


By the bound (\ref{eq:kbound}) along with
Lemmas~\ref{lem:bound} and~\ref{lem:bound2},
a self-dual $\ZZ_{k}$-code $C$  of length $n \le 48$
satisfies the following bound:
\begin{equation}\label{eq:nbound}
d_E(C) \le 
\begin{cases}
3k & \text{if $n =23$ and $k \ge 4$,} \\
4 \lfloor \frac{n}{24} \rfloor+6
  & \text{if $n =22,46$ and $k=2$,} \\  
20 & \text{if $n =47$ and $k=4$,} \\  
2k \left\lfloor \frac{n}{24} \right\rfloor +2k & \text{otherwise.}
\end{cases}
\end{equation}
We say that a self-dual $\ZZ_{k}$-code meeting the bound (\ref{eq:nbound})
with equality is 
{\em extremal}\footnote{For $k=3$, a self-dual code meeting
the bound (\ref{eq:kbound}) is usually called extremal.
However, we here adopt this definition
since we simultaneously consider the existence of extremal
self-dual $\ZZ_k$-codes for all integers $k$
with $k \ge 2$.
}
for length $n \le 48$.
We say that a self-dual $\ZZ_k$-code $C$ is 
{\em near-extremal} if $d_E(C)+k$ meets the bound (\ref{eq:nbound}).
We only consider near-extremal self-dual $\ZZ_k$-codes
when there is no extremal self-dual $\ZZ_k$-code of that length.

The following lemma shows that an extremal self-dual
$\ZZ_{k}$-code of lengths $24$ and $48$ must be Type~II
for every even positive integer $k$.

\begin{lem}
Let $C$ be a Type~I $\ZZ_k$-code of length $n$.
\begin{itemize}
\item[\rm (a)]
If $n=24$, then $d_E(C) \le 3 k$.
\item[\rm (b)]
If $n=48$, then $d_E(C) \le 5 k$.
\end{itemize}
\end{lem}
\begin{proof}
We give a proof of (b).
By the bound (\ref{eq:kbound}),
it is sufficient to consider only $k \ge 4$.
Assume that $d_E(C) \ge 6 k$.
If $k \ge 6$, then $A_k(C)$ has minimum norm $6$
from the upper bound
on the minimum norms of unimodular lattices.
In addition, $A_k(C)$ must be even~\cite{Gaulter}, 
that is, $C$ is Type~II.
If $k=5$, then $A_5(C)$ is an optimal odd unimodular lattice with
kissing number $96$,
which contradicts that
the kissing number is $385024$ or $393216$~\cite{HKMV}.
Finally, suppose that $k=4$.
Since $d_E(C) \ge 24$, 
$A_4(C)$ satisfies the condition that
$\min(A_4(C))=4$, the kissing number is $96$ and there is no
vector of norm $5$.
By~\cite[(2) and (3)]{CS-odd}, one can determine the 
possible theta series of 
$A_4(C)$ and its shadow $S$ as follows:
\[
\begin{cases}
\theta_{A_4(C)}(q) &=
1 + 96 q^4 + (35634176 + 16777216 \alpha)q^6 
+ \cdots, \\
\theta_{S}(q) &=
\alpha + (96 - 96 \alpha )q^2 +(- 4416 + 4512 \alpha) q^4 + \cdots,
\end{cases}
\]
respectively, where $\alpha$ is an integer.
From the coefficients of $q^2$ and $q^4$ in $\theta_{S}(q)$,
it follows that $\alpha=1$.
Hence, 
$A_4(C)$ must be  even, that is, $C$ is Type~II\@.

The proof of (a) is similar to that of (b), and
it can be completed more easily.
So the proof is omitted.
\end{proof}


\begin{rem}
The odd Leech lattice contains a $k$-frame for 
every integer $k$ with $k \ge 3$~\cite{Miezaki}.
The binary odd Golay code is a near-extremal Type~I code of length $24$.
Hence, by Lemma~\ref{lem:LtoC},
there is a near-extremal Type~I $\ZZ_k$-code of length $24$
if and only if $k$ is an integer with $k \ge 2$.
\end{rem}

\subsection{Negacirculant matrices}
\label{subsec:M}

Throughout this paper, let $A^T$ denote the transpose of a matrix $A$ and 
let $I_k$ denote the identity matrix of order $k$.
An  $n \times n$ matrix is
{\em circulant} and {\em negacirculant} if
it has the following form:
\[
\left( \begin{array}{ccccc}
r_0     &r_1     & \cdots &r_{n-2}&r_{n-1} \\
cr_{n-1}&r_0     & \cdots &r_{n-3}&r_{n-2} \\
cr_{n-2}&cr_{n-1}& \ddots &r_{n-4}&r_{n-3} \\
\vdots  & \vdots &\ddots& \ddots & \vdots \\
cr_1    &cr_2    & \cdots&cr_{n-1}&r_0
\end{array}
\right),
\]
where $c=1$ and $-1$, respectively.
Most of matrices constructed in this paper 
are based on negacirculant matrices.
In Section~\ref{sec:main}, 
in order to construct self-dual $\ZZ_k$-codes of length $4n$,
we consider generator 
matrices of the following form:
\begin{equation} \label{eq:GM}
\left(
\begin{array}{ccc@{}c}
\quad & {\Large I_{2n}} & \quad &
\begin{array}{cc}
A & B \\
-B^T & A^T
\end{array}
\end{array}
\right),
\end{equation}
where $A$ and $B$ are $n \times n$ negacirculant matrices.
It is easy to see that the code is self-dual if
$AA^T+BB^T=-I_n$.


In Section~\ref{sec:frame}, 
in order to find $k$-frames in some unimodular lattices, 
we need to construct matrices $M$ satisfying 
the condition (\ref{eq:condition}) 
in Proposition~\ref{prop:const}.
Suppose that $p$ is a prime, which is congruent to $3$ modulo $4$. 
Let $Q_{p}=(q_{ij})$ be a $p \times p$ matrix over $\ZZ$, where
$q_{ij}=0$ if $i=j$, 
$-1$ if $j-i$ is a nonzero square modulo $p$, and 
$1$ otherwise.
We consider the following matrix:
\[
P_{p+1} = 
\left(\begin{array}{ccccccc}
 0       & 1  & \cdots      &1  \\
 -1      & {} & {}          &{} \\
 \vdots  & {} &  Q_{p}  &{} \\
 -1      & {} & {}          &{} \\
\end{array}\right).
\]
Then it is well known that 
$P_{p+1}P_{p+1}^T=pI_{p+1}$,
$P_{p+1}^T=-P_{p+1}$, and 
$P_{p+1}+I_{p+1}$ is a Hadamard matrix of order $p+1$.
Hence, these matrices $P_{p+1}$ satisfy (\ref{eq:condition}).
In Section~\ref{sec:frame}, we 
construct more $2m\times 2m$ matrices $M$ satisfying (\ref{eq:condition})
using the following form:
\begin{equation} \label{eq:3nega}
\left(
\begin{array}{rr}
A_1 & A_2 \\
-A_2^T & A_1^T
\end{array}
\right),
\end{equation}
where $A_1$ and $A_2$ are $m \times m$ negacirculant matrices.
\subsection{Number theoretical results}

In order to give infinite families of $k$-frames 
by Proposition~\ref{prop:const},
the following lemma is needed.
The proofs are given by Miezaki in private communication~\cite{M}, 
which are
similar to those in~\cite{Chapman, HM12, Miezaki}.



\begin{lem}[Miezaki~\cite{M}] 
\label{lem:prime}
\begin{itemize}
\item[\rm (a)]\label{thm:1}
There are integers $a,b,c$ and $d$ satisfying
$b \equiv c-d \pmod 3$,
$d \equiv a+b \pmod 3$ and 
$p=\frac{1}{3}(a^2+25b^2+c^2+25d^2)$ for each prime $p \ne 2, 5, 7, 13, 23$.

\item[\rm (b)] \label{thm:2}
There are integers $a,b,c$ and $d$ satisfying
$b \equiv c-2d \pmod 4$,
$d \equiv a+2b \pmod 4$ and 
$p=\frac{1}{4}(a^2+7b^2+c^2+7d^2)$ for each prime $p \ne 2,7$.

\item[\rm (c)] \label{thm:7}
There are integers $a,b,c$ and $d$ satisfying
$b \equiv c \pmod 5$,
$d \equiv a \pmod 5$ and 
$p=\frac{1}{5}(a^2+49b^2+c^2+49d^2)$ for each prime $p \ne 2, 3, 7, 11,19,29$.

\item[\rm (d)] \label{thm:6}
There are integers $a,b,c$ and $d$ satisfying
$b \equiv c-2d \pmod 5$,
$d \equiv a+2b \pmod 5$ and 
$p=\frac{1}{5}(a^2+25b^2+c^2+25d^2)$ for each prime $p \ne 2, 3, 17$.

\item[\rm (e)] \label{thm:3}
There are integers $a,b,c$ and $d$ satisfying
$b \equiv c-2d \pmod 4$,
$d \equiv a+2b \pmod 4$ and 
$p=\frac{1}{4}(a^2+15b^2+c^2+15d^2)$ for each prime $p \ne 2,3$.

\item[\rm (f)] \label{thm:4}
There are integers $a,b,c$ and $d$ satisfying
$b \equiv c-2d \pmod 6$,
$d \equiv a+2b \pmod 6$ and 
$p=\frac{1}{6}(a^2+49b^2+c^2+49d^2)$ for each prime $p \ne 2, 3, 5,7$.

\item[\rm (g)] \label{thm:5}
There are integers $a,b,c$ and $d$ satisfying
$b \equiv c \pmod 4$,
$d \equiv a \pmod 4$ and 
$p=\frac{1}{4}(a^2+19b^2+c^2+19d^2)$ for each prime $p \ne 2, 3, 13, 19$.

\item[\rm (h)] \label{thm:8}
There are integers $a,b,c$ and $d$ satisfying
$b \equiv c \pmod 5$,
$d \equiv a \pmod 5$ and 
$p=\frac{1}{5}(a^2+39b^2+c^2+39d^2)$ for each prime $p \ne 2, 3, 7, 17$.

\end{itemize}
\end{lem}

\section{Construction of $m$-frames in some unimodular lattices}
\label{sec:frame}

In this section, we provide a method for
constructing $m$-frames in unimodular lattices, which
are constructed from some self-dual $\ZZ_k$-codes by Construction A\@.
Combining Lemma~\ref{lem:prime} with the method, 
we construct $m$-frames in 
odd unimodular lattices.

The following method is a slight generalization 
of~\cite[Propositions 3.3 and 3.6]{HM12}.
Also, the cases 
$(k,m,\ell)=(4,11,2)$ and $(4,11,0)$ of the following method
are used in~\cite{Chapman,Miezaki},
respectively.

\begin{prop}\label{prop:const}
Let $k$ be an integer with $k \ge 2$, and
let $\ell$ be an integer with $0 \le \ell \le k-1$.
Let $M$ be an $n \times n$ 
matrix over $\ZZ$ satisfying 
\begin{equation}\label{eq:condition}
M^T=-M \text{ and } M M^T= mI_n,
\end{equation}
where $m+\ell^2 \equiv -1 \pmod{k}$.
Let $C_{2n,k}(M)$ be the $\ZZ_{k}$-code of length $2n$
with generator matrix
$\left(\begin{array}{cc}
I_n & M+ \ell I_n
\end{array}\right)$, where 
the entries of the matrix are regarded as elements of $\ZZ_{k}$.
Let $a,b,c$ and $d$ be integers with 
$b \equiv c-\ell d \pmod {k}$ and 
$d \equiv a+\ell b \pmod {k}$.
Then $C_{2n,k}(M)$ is self-dual, and
the set of $2n$ rows of the following matrix 
\[
F(M)=
\frac{1}{\sqrt{k}}
\left(
\begin{array}{cc}
aI_n+bM & cI_n+dM \\
-cI_n+dM & aI_n-bM
\end{array}
\right)
\]
forms a $\frac{1}{k}(a^2+m b^2+c^2+m d^2)$-frame in 
the unimodular lattice $A_{k}(C_{2n,k}(M))$.
\end{prop}
\begin{proof}
Since 
$M M^T= mI_n$, $M^T=-M$ and $m+ \ell^2 \equiv -1 \pmod{k}$,
$C_{2n,k}(M)$ is a self-dual $\ZZ_{k}$-code of length $2n$.
Thus, $A_{k}(C_{2n,k}(M))$ is a unimodular lattice.
Since $C_{2n,k}(M)$ is self-dual and $M^T=-M$,
both $G=\left(\begin{array}{cc}
I_n & M+\ell I_n
\end{array}\right)$ 
and
$H=\left(\begin{array}{cc}
M-\ell I_n & I_n
\end{array}\right)$ are  generator matrices of $C_{2n,k}(M)$.

Let $s,t$ be integers.
Here, we regard the entries of the matrices $G,H$ as integers.
Then
\[
\left(\begin{array}{c}
sG+tH \\
-tG+sH
\end{array}\right)
=
\left(\begin{array}{cc}
(s-\ell t)I_n+tM & (\ell s+t)I_n+ sM \\
-(\ell s+t)I_n+sM & (s-\ell t)I_n -tM
\end{array}\right).
\]
Hence, if
$b \equiv c-\ell d \pmod {k}$ and 
$d \equiv a+\ell b \pmod {k}$, then
all rows of the matrix $F(M)$
are vectors of $A_{k}(C_{2n,k}(M))$.
Since $F(M) F(M)^T=\frac{1}{k}(a^2+mb^2+c^2+md^2)I_{2n}$,
the result follows.
\end{proof}

\begin{rem}
\begin{itemize}
\item[(i)]
It follows from the assumption 
that 
$a^2+m b^2+c^2+m d^2 \equiv 0 \pmod k$.
\item[(ii)]
By~\cite[Proposition 2.12]{GS}, 
if $n \equiv 2 \pmod 4$ then $m$ must be a square.
\end{itemize}
\end{rem}

\begin{table}[thbp]
\caption{Matrices satisfying the assumptions in 
Proposition~\ref{prop:const}}
\label{Tab:D}
\begin{center}
{\footnotesize
\begin{tabular}{c|c|l|l}
\noalign{\hrule height0.8pt}
$M$ & $(k,m,\ell)$ 
 & \multicolumn{1}{c|}{$r_{A_1}$} &\multicolumn{1}{c}{$r_{A_2}$} \\
\hline
$D_{6 }$ & (3, 25, 1)  & $(0,  2,  2)$ & $(0,  1, -4)$ \\
$P_{8 }$ & (4,  7, 2)  && \\
$D_{10}$ & (3, 25, 1)  & $(0,0,2,2,0)$&$(1,2,2,-2,2)$ \\
$D'_{10}$ &(3, 25, 1) & $(0,0,0,0,0)$&$(-3, -2,  2, -2,  2)$ \\
$D''_{10}$&(5, 49, 0) & $( 0, 0,3, 3,0)$&$(-2,-3,4,-1,1)$ \\
$D_{14}$ & (3, 25, 1)  & $(0,2,1,0,0,1,2)$&$(-1,-2,1,-2,2,1,0)$ \\
$D'_{14}$ &(5, 25, 2) & $(0,0,2,-1,-1,2,0)$&$(-2,-1,-2,0,-1,-1,-2)$ \\
$D_{16}$ & (4, 15, 2)  & $(0,1,1,0,1,0,1,1)$& $(1,1,1,-1,-1,2,-1,0)$ \\
$D_{18}$ & (6, 49, 2)  & $(0,1,-3,0,2,2,0,-3,1)$ & $(-2,2,-1,2,1,2,1,1,1)$\\
$P_{20}$ & (4, 19, 0)  & & \\
$D_{22}$ & (5, 25, 2)  & 
{$(0,0,-1,1,0,0,0,0,1,-1,0)$} & 
{$(1,0,-2,1,1,1,2,1,0,2,-2)$} \\
$D_{24}$ & (5, 39, 0)  & 
{$( 0, 1, 1, 1, 2,-1, 1,-1, 2, 1, 1, 1)$} &
{$(-2,-1, 2,-1,-1,-2, 0, 1, 0, 2,-1,-1)$} \\
\noalign{\hrule height0.8pt}
\end{tabular}
}
\end{center}
\end{table}

The matrices $P_{p+1}$ $(p=7,19)$, 
which are given in Section~\ref{subsec:M},  
satisfy the assumptions in Proposition~\ref{prop:const},
for the integers $k,m$ and $\ell$ listed in Table~\ref{Tab:D}.
Using the form (\ref{eq:3nega}),
we have found 
matrices $D_n$ $(n=6,10,14,16,18,22,24)$, $D'_{n}$ $(n=10,14)$
and $D''_{10}$
satisfying the assumptions in Proposition~\ref{prop:const},
for the integers $k,m$ and $\ell$ listed in Table~\ref{Tab:D},
where the first rows $r_{A_1}$ and $r_{A_2}$ 
of negacirculant matrices
$A_1$ and $A_2$ in (\ref{eq:3nega}) are listed in Table~\ref{Tab:D}.


By Proposition~\ref{prop:const},
for each of the matrices $M$ given in Table~\ref{Tab:D}, 
the odd unimodular lattice
$A_k(C_{2n,k}(M))$, which is constructed from
the Type~I $\ZZ_k$-code $C_{2n,k}(M)$, contains a
$\frac{1}{k}(a^2+m b^2+c^2+m d^2)$-frame
for integers $a,b,c$ and $d$ with 
$b \equiv c-\ell d \pmod {k}$ and 
$d \equiv a+\ell b \pmod {k}$.
The minimum norms $\min(L)$ of 
the lattices $L=A_k(C_{2n,k}(M))$ listed in Table~\ref{Tab:L},
which have been determined by {\sc Magma},
are also listed in the table.

\begin{lem}\label{lem:key}
Suppose that $L$ is any of 
the lattices listed in Table~\ref{Tab:L}.
Then $L$ contains a $k$-frame
for an integer $k$ satisfying the conditions {\rm ($\star$)}
listed in Table~\ref{Tab:L},
where $m_i$ in {\rm ($\star$)} is a non-negative integer.
\end{lem}
\begin{proof}
All cases are similar, and we only give details for the
lattice $A_3(C_{12,3}(D_6))$.
Let $a,b,c$ and $d$ be integers with 
$b \equiv c-d \pmod 3$ and
$d \equiv a+b \pmod 3$.
By Proposition~\ref{prop:const},
$A_3(C_{12,3}(D_6))$ contains a
$\frac{1}{3}(a^2+25b^2+c^2+25d^2)$-frame.
By Lemma~\ref{lem:prime} (a),
there are integers $a,b,c$ and $d$ satisfying
$b \equiv c-d \pmod 3$,
$d \equiv a+b \pmod 3$ and 
$p=\frac{1}{3}(a^2+25b^2+c^2+25d^2)$ for each prime $p \ne 2, 5, 7, 13, 23$.
The result follows from Lemma~\ref{lem:frame}.
For the other lattices, Table~\ref{Tab:L} lists
the cases of Lemma~\ref{lem:prime}, which are needed in the proof.
\end{proof}

\begin{table}[thbp]
\caption{Unimodular lattices by Proposition~\ref{prop:const}}
\label{Tab:L}
\begin{center}
{\small
\begin{tabular}{c|c|l|c}
\noalign{\hrule height0.8pt}
$L$ & $\min(L)$ & \multicolumn{1}{c|}{Condition ($\star$)} 
&  Case\\
\hline
$A_3(C_{12,3}(D_6))$ &2 &
$k \ge 2$, $k\ne 2^{m_1} 5^{m_2} 7^{m_3}13^{m_4}23^{m_5}$ & (a) \\
$A_4(C_{16,4}(P_{8}))$  &2~\cite{Z4-H40} &
$k \ge 2$, $k\ne 2^{m_1}7^{m_2}$ & (b)\\
$A_3(C_{20,3}(D_{10}))$ &2 &
$k \ge 2$, $k\ne 2^{m_1} 5^{m_2} 7^{m_3}13^{m_4}23^{m_5}$   & (a) \\
$A_3(C_{20,3}(D'_{10}))$ &2 &
$k \ge 2$, $k\ne 2^{m_1} 5^{m_2} 7^{m_3}13^{m_4}23^{m_5}$   & (a) \\
$A_5(C_{20,5}(D''_{10}))$ &2 &
$k \ge 2$, $k\ne 2^{m_1} 3^{m_2} 7^{m_3} 11^{m_4} 19^{m_5} 29^{m_6}$ &(c) \\
$A_3(C_{28,3}(D_{14}))$ &3 &
$k \ge 3$, $k\ne 2^{m_1} 5^{m_2} 7^{m_3}13^{m_4}23^{m_5}$ & (a)  \\
$A_5(C_{28,5}(D'_{14}))$ &3 &
$k \ge 3$, $k\ne 2^{m_1} 3^{m_2} 17^{m_3}$ & (d)\\ 
$A_4(C_{32,4}(D_{16}))$ &4 &
$k \ge 4$, $k\ne 2^{m_1} 3^{m_2}$  & (e) \\
$A_6(C_{36,6}(D_{18}))$ &4 &
$k \ge 4$, $k\ne 2^{m_1} 3^{m_2} 5^{m_3}7^{m_4}$ & (f) \\
$A_4(C_{40,4}(P_{20}))$ &4~\cite{Z4-H40} &
$k \ge 4$, $k\ne 2^{m_1}3^{m_2} 13^{m_3} 19^{m_4}$ & (g)\\
$A_5(C_{44,5}(D_{22}))$ &4 &
$k \ge 4$, $k\ne 2^{m_1} 3^{m_2} 17^{m_3}$ & (d)\\ 
$A_5(C_{48,5}(D_{24}))$ &5 &
$k \ge 5$, $k\ne 2^{m_1} 3^{m_2} 7^{m_3} 17^{m_4}$ & (h)\\ 
\noalign{\hrule height0.8pt}
\end{tabular}
}
\end{center}
\end{table}

\section{Frames of some extremal odd unimodular lattices
and extremal Type~I $\ZZ_k$-codes}
\label{sec:main}

In this section, 
for each $L$ of 
the extremal (optimal) odd unimodular lattices 
listed in Table~\ref{Tab:L}, 
we determine all integers $k$ such that $L$ contains a $k$-frame.
This yields the existence of 
some extremal (near-extremal) Type~I $\ZZ_{k}$-codes.

\subsection{Frames of $D_{12}^+$ and Type~I $\ZZ_k$-codes of length 12}

There is a unique extremal odd unimodular lattice
in dimension $12$, up to isomorphism
(see~\cite[Table~16.7]{SPLAG}),
where the lattice is denoted by $D_{12}^+$.
There is a unique binary extremal Type~I code of length $12$,
up to equivalence~\cite{Pless72b},
where the code is denoted by $B_{12}$ in~\cite[Table~2]{Pless72b}.
It is known that $D_{12}^+ \cong A_2(B_{12})$.
Hence, by Lemma~\ref{lem:key}, it is sufficient to 
investigate the existence of
a $k$-frame in $D_{12}^+$ for $k=5,7,13,23$.

There are 16 inequivalent Type~I $\ZZ_5$-codes of 
length $12$~\cite{LPS-GF5}.
We have verified by {\sc Magma} that
$D_{12}^+ \cong A_5(C_i)$ for $i= 8, 11, 13, 16$, 
where $C_i$ denotes the $i$th code in~\cite[Table~III]{LPS-GF5}.
There are $64$ inequivalent Type~I $\ZZ_7$-codes of 
length $12$~\cite{HO02}, where these codes are denoted by
$C_{12,i}$ ($i=1,2,\ldots,64$) in~\cite[Table~1]{HO02}.
We have verified by {\sc Magma} that
$D_{12}^+ \cong A_7(C_{12,i})$
for $i=11$, $12$, $15$, $17$, $20$, $38$, $42$, $43$, $47$, $49$, $51$, 
$54$, $55$, $57, \ldots, 62$.
For $k=13$ and $23$,
let $C_{k,12}$ be the $\ZZ_k$-code with generator 
matrix of the form (\ref{eq:GM}),
where the first rows $r_A$ and $r_B$ of $A$ and $B$
are as follows:
\[
(r_A,r_B)=
((0,  1,  6),(  2,  3,  1)) \text{ and }
((0,  1, 18),(  7,  4,  0)),
\]
respectively.
Since $AA^T+BB^T=-I_{3}$,
these codes are Type~I\@.
Moreover, we have verified 
by {\sc Magma} that 
$A_k(C_{k,12}) \cong D_{12}^+$ ($k=13,23$).
Hence, combining with Lemma~\ref{lem:key}, we have the following:

\begin{thm}\label{thm:D12}
$D_{12}^+$ contains a $k$-frame
if and only if $k$ is an integer with $k \ge 2$.
\end{thm}

By Lemma~\ref{lem:LtoC}, we have the following:

\begin{cor}\label{cor:12}
There is an extremal Type~I $\ZZ_{k}$-code of
length $12$ 
if and only if $k$ is an integer with $k \ge 2$.
\end{cor}

%


\subsection{Frames of $D_8^2$ and Type~I $\ZZ_k$-codes of length 16}

There is a unique extremal odd unimodular lattice
in dimension $16$, up to isomorphism
(see~\cite[Table~16.7]{SPLAG}),
where the lattice is denoted by $D_8^2$.
There is a unique binary extremal Type~I code of length $16$,
up to equivalence~\cite{Pless72b},
where the code is denoted by $F_{16}$ in~\cite[Table~2]{Pless72b}.
It is known that $D_8^2 \cong A_2(F_{16})$.
Hence, by Lemma~\ref{lem:key}, it is sufficient to 
investigate the existence of
a $7$-frame in $D_8^2$.

Let $C_{7,16}$ be the $\ZZ_7$-code with generator 
matrix of the form (\ref{eq:GM}),
where the first rows $r_A$ and $r_B$ of $A$ and $B$
are as follows:
\[
r_A=(0, 0, 1, 1 ) \text{ and } r_B=(1, 3, 1, 0),  
\]
respectively.
Since $AA^T+BB^T=-I_{4}$,
$C_{7,16}$ is Type~I\@.
We have verified by {\sc Magma} that
$A_7(C_{7,16}) \cong D_8^2$.
Hence, combining with Lemma~\ref{lem:key}, we have the following:

\begin{thm}\label{thm:D82}
$D_8^2$ contains a $k$-frame
if and only if $k$ is an integer with $k \ge 2$.
\end{thm}

By Lemma~\ref{lem:LtoC}, we have the following:

\begin{cor}\label{cor:16}
There is an extremal Type~I $\ZZ_{k}$-code of
length $16$ 
if and only if $k$ is an integer with $k \ge 2$.
\end{cor}

%

\subsection{Frames of $D_4^5, A_5^4, D_{20}$ and
Type~I $\ZZ_k$-codes of length 20}

There are $12$ non-isomorphic extremal odd unimodular lattices in 
dimension $20$ (see~\cite[Table~2.2]{SPLAG}).
We denote the $i$th lattices $(i=1,11,12)$ in dimension $20$
in~\cite[Table~16.7]{SPLAG}, by $D_{20}$, $D_4^5$, $A_5^4$,
respectively.
We have verified by {\sc Magma} that 
$D_4^5 \cong A_3(C_{20,3}(D_{10}))$ in Table~\ref{Tab:L},
$A_5^4 \cong A_3(C_{20,3}(D'_{10}))$ in Table~\ref{Tab:L} and
$D_{20} \cong A_5(C_{20,5}(D''_{10}))$ in Table~\ref{Tab:L}.
By Lemma~\ref{lem:key}, it is sufficient to 
investigate the existence of
a $k$-frame in $D_4^5$ and $A_5^4$ for $k=2,5,7,13,23$, and 
a $k$-frame in $D_{20}$ for $k=2, 3,7,11,19,29$.

\begin{table}[thb]
\caption{Extremal Type~I $\ZZ_{k}$-codes of length $20$}
\label{Tab:20}
\begin{center}
{\footnotesize
\begin{tabular}{c|l|l|c|l|l}
\noalign{\hrule height0.8pt}
Code& \multicolumn{1}{c|}{$r_A$} & \multicolumn{1}{c|}{$r_B$} &
Code& \multicolumn{1}{c|}{$r_A$} & \multicolumn{1}{c}{$r_B$} \\
\hline
$C_{ 5,20}$ &$(0, 0, 0, 1, 1)$ & $( 1, 4, 2, 1, 0)$ &
$C_{ 7,20}$ &$(0, 0, 0, 1, 6)$ & $( 3, 0, 1, 1, 0)$ \\
$C_{13,20}$ &$(0,  0,  0,  1,  1)$ & $( 10,  3,  2,  1,  0)$ &
$C_{23,20}$ &$(0,  0,  0,  1, 18)$ & $(  7,  4,  0,  0,  0)$ \\
\hline
$C'_{ 5,20}$ &$(0, 0, 0, 1, 4)$&$(3, 1, 4, 1, 0)$ &
$C'_{ 7,20}$ &$(0, 0, 0, 1, 5)$&$(1, 5, 3, 1, 0)$\\
$C'_{13,20}$ &$(0, 0, 0, 1, 4)$&$(4, 0, 3, 3, 0)$ &
$C'_{23,20}$ &$(0, 0, 0, 1,12)$&$(3, 5, 7, 1, 0)$\\
\hline
$C''_{ 7,20}$& $(0, 0, 0, 1,  4)$&$(  1,  3,  2, 3, 1)$ &
$C''_{ 9,20}$& $(0,0,0,1,3)$&$(1,2,4,2,6)$ \\
$C''_{11,20}$& $(0, 0, 0, 1,  8)$&$(  5,  6,  6, 3, 2)$ &
$C''_{19,20}$& $(0, 0, 0, 1, 12)$&$( 14, 12, 11, 1, 0)$ \\
$C''_{29,20}$& $(0, 0, 0, 1, 21)$&$(  7, 11, 16, 1, 0)$ & &&\\
\noalign{\hrule height0.8pt}
\end{tabular}
}
\end{center}
\end{table}

\begin{figure}[thb]
\centering
{\small
\begin{tabular}{ll}
$
\left(\begin{array}{cc}
A & B_1+2B_2 \\
\end{array}\right)
=
\left(\begin{array}{cc}
11& 220113303 \\
00& 021012300 \\
10& 222120030 \\
01& 031321330 \\
01& 232220201 \\
11& 231021312 \\
00& 023031002 \\
10& 230133321 \\
11& 333130022 \\
\end{array}\right),
$
&
$
2D=
\left(\begin{array}{c}
202202200 \\
220022200
\end{array}\right)
$
\end{tabular}
\caption{A generator matrix of $C'_{4,20}$}
\label{Fig:20}
}
\end{figure}

There are $7$ binary extremal Type~I codes of length $20$,
up to equivalence~\cite{Pless72b}.
The unique code containing $5$ (resp.\ $45$) codewords of weight $4$ 
is denoted by $M_{20}$ (resp.\ $J_{20}$)
in~\cite[Table~2]{Pless72b}.
We have verified by {\sc Magma} that
$A_2(M_{20}) \cong D_4^5$ and $A_2(J_{20}) \cong D_{20}$.
It is known 
that there is no binary Type~I code $C$ such that
$A_2(C) \cong A_5^4$.
There are $6$ inequivalent ternary self-dual codes of
length $20$ and minimum weight $6$~\cite{PSW}.
We have verified by {\sc Magma} that 
$L$ is an extremal odd unimodular lattice in dimension $20$
such that $L \cong A_3(C)$ for
some ternary self-dual code $C$
if and only if $L$ is isomorphic to $D_4^5$ or $A_5^4$.

Let $C_{k,20}$, $C'_{k,20}$ ($k=5,7,13,23$) and 
$C''_{k,20}$ ($k=7,9,11,19,29$)
be the $\ZZ_k$-codes with generator 
matrices of the form (\ref{eq:GM}),
where the first rows $r_A$ and $r_B$ of $A$ and $B$
are listed in Table~\ref{Tab:20}.
Since $AA^T+BB^T=-I_{5}$,
these codes are Type~I\@.
Let $C'_{4,20}$
be the $\ZZ_4$-code with generator matrix of
the following form:
\[
G_{20}=
\left(\begin{array}{ccc}
I_9 & A & B_1+2B_2 \\
O    &2I_2 & 2D \\
\end{array}\right),
\]
where we list in Figure~\ref{Fig:20}
the matrices
$\left(\begin{array}{cc}
A & B_1+2B_2 \\
\end{array}\right)$
and $2D$.
Here, $A$, $B_1$, $B_2$ and $D$ are $(1,0)$-matrices
and $O$ denotes the zero matrix of appropriate size.
It follows from $G_{20}G_{20}^T=O$ and $\# C'_{4,20}=4^{10}$
that $C'_{4,20}$  is self-dual.
The code $C'_{4,20}$ has been found by 
directly finding a $4$-frame in $A_5^4$ using {\sc Magma}.
In a similar way, some other (new) self-dual $\ZZ_4$-codes 
are also constructed in this paper.
We have verified by {\sc Magma} that
$A_k(C_{k,20}) \cong D_4^5$ ($k=5,7,13,23$),
$A_k(C'_{k,20}) \cong A_5^4$ ($k=4,5,7,13,23$) and
$A_k(C''_{k,20}) \cong D_{20}$ ($k=7,9,11,19,29$).
Hence, combining with Lemma~\ref{lem:key}, we have the following:

\begin{thm}\label{thm:20}
$D_4^5$ contains a $k$-frame 
if and only if $k$ is an integer with $k \ge 2$.
$A_5^4$ contains a $k$-frame 
if and only if $k$ is an integer with $k \ge 3$.
$D_{20}$ contains a $k$-frame 
if and only if $k$ is an integer with $k \ge 2$, $k \ne 3$.
\end{thm}

\begin{rem}
$D_{20}$ has theta series
$1 + 760 q^2 + 77560 q^4 + 524288 q^5 +\cdots$.
\end{rem}

By Lemma~\ref{lem:LtoC}, we have the following:

\begin{cor}\label{cor:20}
There is an extremal Type~I $\ZZ_{k}$-code of
length $20$ 
if and only if $k$ is an integer with $k \ge 2$.
\end{cor}


%
%

\subsection{Type~I $\ZZ_k$-codes of length 28}

There is no extremal odd unimodular lattice in dimension $28$
and the largest minimum norm among
odd unimodular lattices in dimension $28$ is $3$.
There are $38$ non-isomorphic optimal odd unimodular
lattices in dimension $28$~\cite{BV}.
In~\cite{BV}, the $38$ lattices are 
denoted by
${\mathbf{R}_{28,1}}(\emptyset)$,
${\mathbf{R}_{28,2}}(\emptyset),\ldots,
{\mathbf{R}_{28,36}}(\emptyset)$,
${\mathbf{R}_{28,37e}}(\emptyset)$,
${\mathbf{R}_{28,38e}}(\emptyset)$.
We have verified by {\sc Magma} that 
${\mathbf{R}_{28,32}}(\emptyset) \cong 
A_3(C_{28,3}(D_{14}))$ in Table~\ref{Tab:L}
and ${\mathbf{R}_{28,15}}(\emptyset) \cong 
A_5(C_{28,5}(D'_{14}))$ in Table~\ref{Tab:L}.
By Lemma~\ref{lem:key}, it is sufficient to 
investigate the existence of
a $k$-frame in ${\mathbf{R}_{28,32}}(\emptyset)$
for $k=4,5,7,13,23$ and 
a $k$-frame in ${\mathbf{R}_{28,15}}(\emptyset)$
for $k=3,4,17$. 

\begin{table}[thb]
\caption{Near-extremal Type~I $\ZZ_{k}$-codes of length $28$}
\label{Tab:28}
\begin{center}
{\footnotesize
\begin{tabular}{c|l|l}
\noalign{\hrule height0.8pt}
Code & \multicolumn{1}{c|}{$r_A$} & \multicolumn{1}{c}{$r_B$} \\
\hline
$C_{ 5,28}$ &$(0,0,0,1,3,4,2)$ & $(3,1,2,0,3,4,0)$\\
$C_{ 7,28}$ &$(0,1,2,2,4,2,3)$ & $(2,2,4,0,4,1,2)$\\
$C_{13,28}$ &$(0, 0, 0, 1, 0, 9, 1)$ & $(5, 1, 3, 7, 7, 1, 4)$\\
$C_{23,28}$ &$(0, 0, 0, 1,12, 1, 1)$ & $(3,19, 7, 5,14,21,17)$\\
\hline
$C'_{17,28}$ & $( 0,0,0,1,13,14, 2)$&  $(10,1,1,9,16,11,15)$ \\
\noalign{\hrule height0.8pt}
\end{tabular}
}
\end{center}
\end{table}

Let $C_{k,28}$ ($k=5,7,13,23$) and $C'_{17,28}$
be the $\ZZ_k$-codes with generator 
matrices of the form (\ref{eq:GM}),
where the first rows $r_A$ and $r_B$ of $A$ and $B$
are listed in Table~\ref{Tab:28}.
Since $AA^T+BB^T=-I_{7}$,
these codes are Type~I\@.
Let 
$C_{4,28}$ and $C'_{4,28}$  
be the $\ZZ_4$-codes with generator matrices of
the following form:
\[
\left(\begin{array}{ccc}
I_{13} & A & B_1+2B_2 \\
O    &2I_{2} & 2D \\
\end{array}\right),
\]
where we list in Figure~\ref{Fig}
the matrices
$
\left(\begin{array}{cc}
A & B_1+2B_2 \\
2I_{2} & 2D \\
\end{array}\right).
$
Then these codes are Type~I\@.
For $k=4,5,7,13,23$,
we have verified by {\sc Magma} that
$A_k(C_{k,28}) \cong {\mathbf{R}_{28,32}}(\emptyset)$.
For $k=4,17$,
we have verified by {\sc Magma} that
$A_k(C'_{k,28}) \cong {\mathbf{R}_{28,15}}(\emptyset)$.
It is known that
${\mathbf{R}_{28,15}}(\emptyset)$ contains a $3$-frame
(see~\cite{HMV} for a classification of $3$-frames in 
the $38$ lattices).
Hence, combining with Lemma~\ref{lem:key}, we have the following:

\begin{thm}\label{thm:28}
${\mathbf{R}_{28,i}}(\emptyset)$ $(i=15,32)$
contains a $k$-frame 
if and only if $k$ is an integer with $k \ge 3$.
\end{thm}

\begin{lem}
Let $C$ be a Type~I $\ZZ_k$-code of length $28$.
Then $d_E(C) \le 3k$.
\end{lem}
\begin{proof}
As described above,
the largest minimum norm among
odd unimodular lattices in dimension $28$ is $3$.
Since $d_E(C) \le 3k$ for $k=2,3$ (see~\cite{CPS, MPS}),
it is sufficient to consider the cases $k \ge 4$ only.
Assume that $d_E(C)  \ge 4k$.
Since $\min(A_{k}(C))=\min\{k, d_{E}(C)/k\}$,
$\min(A_{k}(C)) \ge 4$, which is a contradiction.
\end{proof}

There are three inequivalent binary Type~I codes of length $28$
and minimum weight $6$~\cite{CPS}.
Hence, by Lemma~\ref{lem:LtoC}, we have the following:

\begin{cor}\label{cor:28}
There is  a near-extremal Type~I $\ZZ_{k}$-code of
length $28$ 
if and only if $k$ is an integer with $k \ge 2$.
\end{cor}

%

\begin{figure}[thb]
\centering
{\footnotesize
$
\left(\begin{array}{cc}
00&3221032113010\\
00&2312302202000\\
01&1011113132031\\
01&2021011201031\\
10&3033332032202\\
00&2220031132311\\
00&1130232110223\\
11&2213122020013\\
01&3200201111201\\
01&3133230220230\\
10&3111000202123\\
10&3011332120200\\
10&3331011112112 \\
20&2220022000000 \\
02&0022222000000
\end{array}\right) \text{ and }
\left(\begin{array}{cc}
01& 1023301203302\\
01& 1022200000021\\
01& 1130203022312\\
11& 1202303012212\\
00& 2321232113032\\
01& 0002112332213\\
11& 1113323310300\\
00& 3321000023111\\
00& 1210231221321\\
11& 2012013002211\\
11& 1010001123020\\
11& 2203101320001\\
00& 3302011030033\\
20& 0002202020200 \\
02& 2220222202200
\end{array}\right) 
$
\caption{Generator matrices of $C_{4,28}$ and $C'_{4,28}$}
\label{Fig}
}
\end{figure}

\subsection{Type~I $\ZZ_k$-codes of length 32}
\label{subsec:32}

There are $5$ non-isomorphic extremal odd unimodular lattices in 
dimension $32$, and 
these $5$ lattices are related to the $5$ inequivalent 
binary extremal Type~II codes of length $32$~\cite{CS-odd}.
The $5$ codes are denoted by 
$\text{C}81,\text{C}82,\ldots,\text{C}85$ 
in~\cite[Table~A]{CPS}.
We denote the extremal odd unimodular lattice
related to $\text{C}i$ by $L_{32,i}$ ($i=81,\ldots,85$).
We have verified 
by {\sc Magma} that $L_{32,82} \cong A_4(C_{32,4}(D_{16}))$  
in Table~\ref{Tab:L}.
Since $A_4(C_{32,4}(D_{16}))$ contains a $4$-frame,
it is sufficient to investigate the existence of
a $k$-frame in $L_{32,82}$ for $k=6,9$ by Lemma~\ref{lem:key}.

\begin{table}[thb]
\caption{Extremal Type~I $\ZZ_{k}$-codes of length $32$}
\label{Tab:32}
\begin{center}
{\footnotesize
\begin{tabular}{c|l|l}
\noalign{\hrule height0.8pt}
Code & \multicolumn{1}{c|}{$r_A$} & \multicolumn{1}{c}{$r_B$} \\
\hline
$C_{6,32}$ & $(0,0,1,2,2,2,1,2)$ & $(1,0,5,5,1,1,3,3)$ \\
$C_{9,32}$ & $(0,0,1,5,0,6,0,1)$ & $(0,6,2,2,7,6,1,7)$ \\
\noalign{\hrule height0.8pt}
\end{tabular}
}
\end{center}
\end{table}

For $k=6,9$, let $C_{k,32}$ be the
$\ZZ_k$-code  with generator matrix of the form (\ref{eq:GM}),
where the first rows $r_A$ and $r_B$ of $A$ and $B$
are listed in Table~\ref{Tab:32}.
Since $AA^T+BB^T=-I_{8}$, these codes are self-dual.
Note that the code $C_{6,32}$ is not Type~II,
since $\wt_E(r_A)+\wt_E(r_B)=41$, 
where $\wt_E(x)$ denotes the Euclidean weight of $x$.
For $k=6,9$, we have verified by {\sc Magma} that
$A_k(C_{k,32}) \cong L_{32,82}$.
Hence, combining with Lemma~\ref{lem:key}, we have the following:

\begin{thm}\label{thm:32}
$L_{32,82}$ contains a $k$-frame
if and only if $k$ is an integer with $k \ge 4$.
\end{thm}

There are three inequivalent
binary extremal Type~I codes of length $32$~\cite{C-S}.
Any ternary self-dual code of length $32$ has minimum weight 
at most $9$~\cite{MS73}.
Hence, by Lemma~\ref{lem:LtoC}, we have the following:

\begin{cor}\label{cor:32}
There is an extremal Type~I $\ZZ_{k}$-code of
length $32$ 
if and only if $k$ is a positive integer with $k \ne 1,3$.
\end{cor}

%

For each extremal odd unimodular lattice in dimension $32$,
one of the even unimodular neighbors is 
extremal~\cite{CS-odd}.
Moreover, it follows from the construction in~\cite{CS-odd}
that the extremal even unimodular neighbor of $L_{32,82}$ 
is the $32$-dimensional Barnes--Wall lattice $BW_{32}$
(see e.g.~\cite[Chapter~8, Section~8]{SPLAG} for $BW_{32}$).
Since the even sublattice of $L_{32,82}$ contains
a $2k$-frame for every integer $k$ with $k \ge 2$
by Theorem~\ref{thm:32},
we have the following:

\begin{prop}
$BW_{32}$ contains a $2k$-frame
if and only if $k$ is an integer with $k \ge 2$.
\end{prop}

Since there are $5$ inequivalent binary extremal Type~II codes of 
length $32$~\cite{CPS},
by Lemma~\ref{lem:LtoC}, we have an alternative proof of the following:

\begin{cor}[Harada and Miezaki~\cite{HM12}]
There is an extremal Type~II $\ZZ_{2k}$-code of
length $32$ 
if and only if $k$ is a positive integer.
\end{cor}

\subsection{Type~I $\ZZ_k$-codes of length 36}

Since $A_6(C_{36,6}(D_{18}))$ in Table~\ref{Tab:L} contains a $6$-frame,
it is sufficient to investigate the existence of
a $k$-frame in $A_6(C_{36,6}(D_{18}))$ for $k=4,5,7,9$
by Lemma~\ref{lem:key}.
For $k=5,7,9$, let $C_{k,36}$ be the
$\ZZ_k$-code  with generator matrix of the form (\ref{eq:GM}),
where the first rows $r_A$ and $r_B$ of $A$ and $B$
are listed in Table~\ref{Tab:36}.
Since $AA^T+BB^T=-I_{9}$,
these codes are Type~I\@.
Let $C_{4,36}$ be the $\ZZ_4$-code with generator matrix of
the following form:
\[
\left(\begin{array}{ccc}
I_{16} & A & B_1+2B_2 \\
O    &2I_{4} & 2D \\
\end{array}\right),
\]
where we list in Figure~\ref{Fig:36}
the matrices
$\left(\begin{array}{cc}
A & B_1+2B_2 \\
\end{array}\right)$
and $2D$.
It follows that $C_{4,36}$ is self-dual.
For $k=4,5,7,9$, 
we have verified by {\sc Magma} that
$A_k(C_{k,36}) \cong A_6(C_{36,6}(D_{18}))$.
Hence, combining with Lemma~\ref{lem:key}, we have the following:

\begin{thm}\label{thm:36}
$A_6(C_{36,6}(D_{18}))$ contains a $k$-frame 
if and only if $k$ is an integer with $k \ge 4$.
\end{thm}
\begin{rem}
We have verified by {\sc Magma} that $A_6(C_{36,6}(D_{18}))$
has theta series $1 + 42840q^4 + 1916928q^5 +
\cdots$ and automorphism group of order $288$
(see~\cite{H36} for details to distinguish $A_6(C_{36,6}(D_{18}))$
from the known lattices, and construction
of more extremal odd unimodular lattices).
\end{rem}

\begin{table}[thb]
\caption{Extremal Type~I $\ZZ_{k}$-codes of length $36$}
\label{Tab:36}
\begin{center}
{\footnotesize
\begin{tabular}{c|l|l}
\noalign{\hrule height0.8pt}
Code & \multicolumn{1}{c|}{$r_A$} & \multicolumn{1}{c}{$r_B$} \\
\hline
$C_{5,36}$ & $(0,1,1,2,3,2,0,2,3)$ & $(1,1,0,2,0,3,4,0,4)$ \\
$C_{7,36}$ & $(0,1,6,2,3,3,6,4,5)$ & $(4,3,3,6,2,4,3,0,3)$ \\
$C_{9,36}$ & $(0,1,0,5,5,0,0,0,3)$ & $(0,2,3,3,4,5,5,7,3)$ \\
\noalign{\hrule height0.8pt}
\end{tabular}
}
\end{center}
\end{table}

\begin{figure}[thb]
\centering
{\footnotesize
\begin{tabular}{ll}
$
\left(\begin{array}{cc}
A & B_1+2B_2 \\
\end{array}\right)
=
\left(\begin{array}{cc}
0100 & 1203131221301121 \\
1011 & 1011202100200000 \\
1010 & 2020221222311322 \\
0101 & 1311223022101123 \\
1110 & 0022223222133220 \\
0110 & 0102101300313130 \\
0100 & 1003131232103103 \\
0001 & 0212210231101002 \\
1101 & 3311103322131110 \\
0101 & 3033123233020103 \\
0101 & 1320133200323130 \\
0100 & 2002221022321133 \\
0101 & 3211333002312322 \\
0101 & 1031113220233320 \\
0101 & 0103111200301112 \\
1110 & 2222020200331300 
\end{array}\right),
$
&
$
2D=
\left(\begin{array}{c}
0020020000222022\\
0200202000000220\\
2222220000020000\\
2022000000000000
\end{array}\right)
$
\end{tabular}
\caption{A generator matrix of $C_{4,36}$}
\label{Fig:36}
}
\end{figure}

There are $41$ inequivalent binary extremal Type~I codes of 
length $36$~\cite{MG08}.
There is a ternary extremal Type~I code of length $36$~\cite{Pless72}.
Hence, by Lemma~\ref{lem:LtoC}, we have the following:

\begin{cor}\label{cor:36}
There is an extremal Type~I $\ZZ_{k}$-code of
length $36$ 
if and only if $k$ is an integer with $k \ge 2$.
\end{cor}

%

\subsection{Type~I $\ZZ_k$-codes of length 40}
\label{subsec:40}

Since $A_4(C_{40,4}(P_{20}))$ in Table~\ref{Tab:L} contains a $4$-frame,
it is sufficient to  investigate the existence of
a $k$-frame in extremal odd unimodular lattices in 
dimension $40$ for $k=6,9,13,19$ by Lemma~\ref{lem:key}.
For $k=9,13,19$,
let $C_{k,40}$ be the $\ZZ_k$-code with generator 
matrix of the form (\ref{eq:GM}),
where the first rows $r_A$ and $r_B$ of $A$ and $B$
are listed in Table~\ref{Tab:40}.
Since $AA^T+BB^T=-I_{10}$, these codes are Type~I\@.
Moreover, we have verified 
by {\sc Magma} that $A_k(C_{k,40})$ is extremal ($k=9,13,19$).
An extremal Type~I $\ZZ_{6}$-code of length $40$ can be
found in~\cite{GH05}.
Hence, combining with Lemma~\ref{lem:key}, we have the following:

\begin{prop}
There is an extremal odd unimodular lattice in dimension $40$
containing a $k$-frame if and only if
$k$ is an integer with $k \ge 4$.
\end{prop}

\begin{rem}
The possible theta series of an extremal odd unimodular
lattice in dimension $40$ is given in~\cite{BBH}:
$\theta_{40,\alpha}(q)=1 + (19120  + 256 \alpha) q^4
+ (1376256 - 4096 \alpha) q^5 + \cdots$,
where $\alpha$ is even with $0 \le \alpha \le 80$.
There are $16470$ non-isomorphic 
extremal odd unimodular lattices in dimension $40$ 
having theta series $\theta_{40,80}(q)$, and these lattices
are related to the $16470$ inequivalent binary extremal
Type~II codes of length $40$~\cite{BBH}.
We have verified by {\sc Magma} that $A_4(C_{40,4}(P_{20}))$ 
has theta series $\theta_{40,80}(q)$
and automorphism group of order $7172259840$, and 
$A_k(C_{k,40})$ ($k=9,13,19$) 
have theta series $\theta_{40,0}(q)$
and automorphism group of order $40$.
Also, we have verified by {\sc Magma} that
three lattices $A_k(C_{k,40})$ ($k=9,13,19$) 
are non-isomorphic.
\end{rem}

There are $10200655$ inequivalent binary extremal Type~I 
codes of length $40$~\cite{BBH}.
There is a ternary extremal Type~I code of length $40$
(see~\cite[Table~XII]{RS-Handbook}).
Hence, by Lemma~\ref{lem:LtoC}, we have the following:

\begin{cor}
There is an extremal Type~I $\ZZ_{k}$-code of
length $40$ 
if and only if $k$ is an integer with $k \ge 2$.
\end{cor}

\begin{table}[thb]
\caption{Extremal Type~I $\ZZ_{k}$-codes of length $40$}
\label{Tab:40}
\begin{center}
{\footnotesize
\begin{tabular}{c|l|l}
\noalign{\hrule height0.8pt}
Code & \multicolumn{1}{c|}{$r_A$} & \multicolumn{1}{c}{$r_B$} \\
\hline
$C_{ 9,40}$ &$(0,0,1,0,5,8,3,0,4,4)$ & $(0,5,0,0,5,6,7,2,5,8)$ \\
$C_{13,40}$ &$(0,0,1,4,10, 5, 1,10,11, 4)$ & $(11,4, 4,6, 7,12,11, 7, 2,8)$\\
$C_{19,40}$ &$(0,0,1,2,14,16,17, 1, 0,13)$ & $(10,2,15,2,18,16, 9,15,12,0)$\\
\noalign{\hrule height0.8pt}
\end{tabular}
}
\end{center}
\end{table}

We have verified by {\sc Magma} that
at least one of the even unimodular neighbors of
$L$ is extremal 
for $L=A_4(C_{40,4}(P_{20}))$, $A_9(C_{9,40})$,
$A_{13}(C_{13,40})$ and $A_{19}(C_{19,40})$.
There are $16470$
inequivalent binary extremal Type~II codes of length $40$~\cite{BHM}.
By Lemma~\ref{lem:LtoC}, we have an alternative proof of the following:

\begin{cor}[Harada and Miezaki~\cite{HM12}]
There is an extremal Type~II $\ZZ_{2k}$-code of
length $40$ 
if and only if $k$ is a positive integer.
\end{cor}

\subsection{Type~I $\ZZ_k$-codes of length 44}

By considering $A_5(C_{44,5}(D_{22}))$ in Table~\ref{Tab:L},
it is sufficient to investigate the existence of
a $k$-frame in extremal odd unimodular lattices in 
dimension $44$ for $k=4,6,9,17$ by Lemma~\ref{lem:key}.
For $k=9,17$,
let $C_{k,44}$ be the $\ZZ_k$-code with generator 
matrix of the form (\ref{eq:GM}),
where the first rows $r_A$ and $r_B$ of $A$ and $B$
are listed in Table~\ref{Tab:44}.
Since $AA^T+BB^T=-I_{11}$, 
these codes are Type~I\@.
Moreover, we have verified 
by {\sc Magma} that $A_k(C_{k,44})$ is extremal ($k=9,17$).
For $k=4$ and $6$, 
an extremal Type~I $\ZZ_{k}$-code of length $44$ can be
found in~\cite[Table~1]{H12} and~\cite{GH05}, respectively.
Hence, combining with Lemma~\ref{lem:key}, we have the following:

\begin{prop}
There is an extremal odd unimodular lattice in dimension $44$
containing a $k$-frame if and only if 
$k$ is an integer with $k \ge 4$.
\end{prop}

\begin{rem}
The possible theta series of an extremal odd unimodular
lattice in dimension $44$ is given in~\cite{H03}:
$\theta_{44,1,\beta}(q)=1+(6600+16\beta)q^4+(811008-128\beta)q^5 + \cdots$,
$\theta_{44,2,\beta}(q)=1+(6600+16\beta)q^4+(679936-128\beta)q^5 + \cdots$,
where $\beta$ is an integer.
We have verified by {\sc Magma} that 
$A_5(C_{44,5}(D_{22}))$ and 
$A_k(C_{k,44})$ ($k=9,17$)
have theta series 
$\theta_{44,1,\beta}(q)$ ($\beta=0,88,176$) and automorphism groups 
of orders $44$, $88$, $44$, respectively.
\end{rem}
\begin{table}[thb]
\caption{Extremal Type~I $\ZZ_{k}$-codes of length $44$}
\label{Tab:44}
\begin{center}
{\footnotesize
\begin{tabular}{c|l|l}
\noalign{\hrule height0.8pt}
Code & \multicolumn{1}{c|}{$r_A$} & \multicolumn{1}{c}{$r_B$} \\
\hline
$C_{9 ,44}$& $(0,0,0,0,1,0,1,4,0,8,0)$ & $(7,0,7,1,8,8,2,8,1,5,1)$ \\
$C_{17,44}$& $(0, 0, 0, 0, 1,13, 7,13,11,16,13)$ 
    &$(12,14, 8,14, 7,12,14, 7,14,14, 7)$ \\
\noalign{\hrule height0.8pt}
\end{tabular}
}
\end{center}
\end{table}

For $k=2,3$, 
there is an extremal Type~I $\ZZ_k$-code of length $44$
(see~\cite[Tables X and XII]{RS-Handbook}).
Hence, by Lemma~\ref{lem:LtoC}, we have the following:

\begin{cor}
There is an extremal Type~I $\ZZ_{k}$-code of
length $44$ 
if and only if $k$ is an integer with $k \ge 2$.
\end{cor}

\subsection{Remarks on Type~I $\ZZ_k$-codes of length 48}\label{sec:rem}

By considering $A_5(C_{48,5}(D_{24}))$ in Table~\ref{Tab:L},
we examine the existence of
a $k$-frame in optimal odd unimodular lattices in 
dimension $48$ for $k=6, 7, 8, 9, 17$ by Lemma~\ref{lem:key}.
It was shown in~\cite{HKMV} that
an extremal even unimodular lattice in dimension $48$
has an optimal odd unimodular neighbor.
Using this result, we have the following:

\begin{lem}\label{lem:6-1}
There is an optimal odd unimodular lattice in dimension $48$
containing an $8k$-frame for every positive integer $k$.
\end{lem}
\begin{proof}
Let $\Lambda$ be 
an extremal even unimodular lattice in dimension $48$.
Let $x$ be a vector of $\Lambda$ with $(x,x)=8$.
Note that there are vectors of norm $8$ in  $\Lambda$
(see~\cite[Chap.~7, (68)]{SPLAG}).
Put $\Lambda_x^{+}=\{v\in\Lambda\mid(x,v)\equiv0\pmod2\}$.
Since there is a vector $y$ of $\Lambda$ such
that $(x,y)$ is odd, the following lattice
\[
\Lambda_{x}=
\Lambda_x^+ \cup \Big(\frac{1}{2}x+y\Big)+\Lambda_x^+
\]
is an optimal odd unimodular neighbor of $\Lambda$~\cite{HKMV}.

Some extremal even unimodular lattice in
dimension $48$ containing an $8$-frame can be found 
in~\cite[Corollary~1]{Chapman-Sole}.
We take this lattice as $\Lambda$ in the above construction.
Let $\{f_1, \ldots, f_{48}\}$ be an $8$-frame in $\Lambda$.
Then $\Lambda_{f_1}$ is an optimal odd unimodular neighbor
containing $\{f_1, \ldots, f_{48}\}$.
The result follows from Lemma~\ref{lem:frame}.
\end{proof}

\begin{table}[thb]
\caption{Near-extremal Type~I $\ZZ_{k}$-codes of length $48$}
\label{Tab:48}
\begin{center}
{\footnotesize
\begin{tabular}{c|l|l}
\noalign{\hrule height0.8pt}
Code & \multicolumn{1}{c|}{$r_A$} & \multicolumn{1}{c}{$r_B$} \\
\hline
$C_{7 ,48}$& $(0,1,6,3,0,2,0,2,4,2,5,3)$&$(3,6,1,5,4,6,0,5,0,5,1,5)$\\
$C_{9 ,48}$& $(0,1,2,4,6,1,6,2,2,0,3,0)$&$(7,2,5,1,6,8,4,1,2,2,8,4)$\\
\noalign{\hrule height0.8pt}
\end{tabular}
}
\end{center}
\end{table}

Some near-extremal Type~I $\ZZ_6$-code $C_{6,48}$ of length $48$
can be found in~\cite{HKMV}.
For $k=7,9$,
let $C_{k,48}$ be the $\ZZ_k$-code with generator 
matrix of the form (\ref{eq:GM}),
where the first rows $r_A$ and $r_B$ of $A$ and $B$
are listed in Table~\ref{Tab:48}.
Since $AA^T+BB^T=-I_{12}$, 
these codes are Type~I\@.
Moreover, we have verified 
by {\sc Magma} that $A_k(C_{k,48})$ is optimal ($k=7,9$).
Hence, we have the following:

\begin{prop}
There is an optimal odd unimodular lattice in dimension $48$
containing a $k$-frame 
for every integer $k$ with $k \ge 5$ and
$k \ne 2^{m_1}3^{m_2}17^{m_3}$,
where 
$m_i$ are integers $(i=1,2,3)$ with 
$(m_1,m_2) \in \{(0,0),(0,1),(1,0),(2,0)\}$ and $m_3 \ge 1$.
\end{prop}

\begin{rem}
$A_6(C_{6,48})$ has kissing number $393216$
\cite[p.~553]{HKMV}.
In addition, we have verified by {\sc Magma} that
$A_5(C_{48,5}(D_{24}))$,
$A_7(C_{7,48})$ and
$A_9(C_{9,48})$ have
kissing number $393216$.
\end{rem}


For $k=2,3$, 
there is a near-extremal Type~I $\ZZ_k$-code of length $48$
(see~\cite[Tables X and XII]{RS-Handbook}).
Let $C_{4,48}$ be the $\ZZ_4$-code with generator matrix:
\[
\left(
\begin{array}{ccc@{}c}
\quad & {\Large I_{24}} & \quad &
\begin{array}{cccc}
0 & 1  & \cdots & 1 \\
1 & {}     & {}     &{} \\
\vdots & {}     & R      &{} \\
1 & {}     &{}      &{} \\
\end{array}
\end{array}
\right),
\]
where $R$ is the $23 \times 23$
circulant matrix with first row 
\[
(1,1,3,0,3,3,1,2,0,1,3,2,3,0,0,3,3,2,1,2,1,1,0). 
\]
We have verified by {\sc Magma} that $C_{4,48}$
is a near-extremal Type~I $\ZZ_4$-code of length $48$.
Hence, by Lemma~\ref{lem:LtoC}, we have the following:

\begin{cor}
There is a near-extremal Type~I $\ZZ_{k}$-code of length $48$
for $k=2,3,4$ and 
for integers $k$ with $k \ge 5$,
$k \ne 2^{m_1}3^{m_2}17^{m_3}$,
where 
$m_i$ are integers $(i=1,2,3)$ with 
$(m_1,m_2) \in \{(0,0),(0,1),(1,0),(2,0)\}$ and $m_3 \ge 1$.
\end{cor}

Using the method in the previous subsections, we tried to construct a
near-extremal Type~I $\ZZ_{17}$-code of length $48$.
However, our extensive search failed to discover such
a code, then we stopped our search at length $48$.
It is worthwhile to determine whether there is a
near-extremal Type~I $\ZZ_{17}$-code of length $48$.

\bigskip
\noindent {\bf Acknowledgments.}
The author would like to thank Tsuyoshi Miezaki for useful discussions,
for carefully reading the manuscript 
and for kindly providing the results given in Lemma~\ref{lem:prime},
and Masaaki Kitazume and Hiroki Shimakura
for helpful conversations.
The author would also like to thank the anonymous referees for
valuable comments leading to several improvements of the presentation.
This work is supported by JSPS KAKENHI Grant Number 23340021.


\end{document}